\documentclass[a4paper, 11pt]{article}
\usepackage{authblk, setspace, subfig, caption}
\usepackage{amsmath}
\usepackage{amssymb,amscd, mathtools} 
\usepackage{enumerate} 
\usepackage[active]{srcltx} 
\usepackage{amsthm, leftidx}
\usepackage{multirow, booktabs}
\usepackage{algorithm}
\usepackage{algpseudocode}

\theoremstyle{definition}
\newtheorem{definition}{Definition}

\newtheorem{example}{Example}

\theoremstyle{plain}
\newtheorem{theorem}{Theorem}
\newtheorem{lemma}{Lemma}
\newtheorem{proposition}{Proposition}
\newtheorem{corollary}{Corollary}

\usepackage{tikz}
\usetikzlibrary{mindmap,shadows, calc}
\usepackage{pgfplots, pgfplotstable}
\makeatletter
\long\def\ifnodedefined#1#2#3{%
	\@ifundefined{pgf@sh@ns@#1}{#3}{#2}%
}

\pgfplotsset{
	discontinuous/.style={
		scatter,
		scatter/@pre marker code/.code={
			\ifnodedefined{marker}{
				\pgfpointdiff{\pgfpointanchor{marker}{center}}%
				{\pgfpoint{0}{0}}%
				\ifdim\pgf@y>0pt
				\tikzset{options/.style={mark=*, fill=white}}
				\draw [densely dashed] (marker-|0,0) -- (0,0);
				\draw plot [mark=*] coordinates {(marker-|0,0)};
				\else
				\tikzset{options/.style={mark=none}}
				\fi
			}{
				\tikzset{options/.style={mark=none}}        
			}
			\coordinate (marker) at (0,0);
			\begin{scope}[options]
			},
			scatter/@post marker code/.code={\end{scope}}
	}
}

\makeatother

\usepackage[hidelinks,pdfencoding=auto]{hyperref}
\usepackage{siunitx}
\newcommand{\RR}{\mathbb{R}}

\def\states{\mathcal X}

\def\cset{\mathcal M}
\def\csete{\mathcal C}
\def\allgambles{\RR^n}

\def\lpr{\underline P}
\def\upr{\overline P}

\def\gambleset{\mathcal F}

\usepackage[numbers]{natbib} 
\newcommand{\low}[1]{{\underline{#1}}}
\newcommand{\up}[1]{{\overline{#1}}}
\newcommand{\ncone}[2]{N(#1, #2)}

\newcommand{\charf}[1]{{\mathbb I_{#1}}}
\newcommand{\setcol}[1]{\mathcal #1}
\newcommand{\cone}[1]{\mathrm{cone}(#1)}
\newcommand{\ri}[1]{\mathrm{ri}(#1)}

\title{A complete characterization of normal cones and extreme points for $p$-boxes}
\author[1]{Damjan \v{S}kulj \\ University of Ljubljana, Faculty of Social Sciences \\ Kardeljeva pl. 5, SI-1000 Ljubljana, Slovenia \\ \href{mailto:damjan.skulj@fdv.uni-lj.si}{\tt damjan.skulj@fdv.uni-lj.si} }

\begin{document}
	
	\maketitle

	\begin{abstract}
		Probability boxes, also known as $p$-boxes, correspond to sets of probability distributions bounded by a pair of distribution functions. They fall into the class of models known as imprecise probabilities. One of the central questions related to imprecise probabilities are the intervals of values corresponding to expectations of random variables, and especially the interval bounds. In general, those are attained in extremal points of credal sets, which denote convex sets of compatible probabilistic models. The aim of this paper is a characterization and identification of extreme points corresponding to $p$-boxes on finite domains. To accomplish this, we utilize the concept of normal cones. In the settings of imprecise probabilities, those correspond to sets of random variables whose extremal expectations are attained in a common extreme point. Our main results include a characterization all possible normal cones of $p$-boxes, their relation with extreme points, and the identification of adjacency structure on the collection of normal cones, closely related to the adjacency structure in the set of extreme points.  
		
		
	\end{abstract}

	\section{Introduction}
	Probability boxes or $p$-boxes \cite{ferson2003} for short, including generalized $p$-boxes \cite{destercke2008unifying}, are models belonging to a wider family of imprecise probabilistic models. Thus being capable of modelling situations that no single probabilistic model can adequately describe (see e.g. \cite{augustin2014introduction} and references therein for further sources on general models of imprecise probabilities). A $p$-box is given as a pair of distribution functions $(\low F, \up F)$ giving rise to a set of distribution functions $F$ such that $\low F\le F \le \up F$, usually interpreted as the set of models compatible with the available information. This set is also called a \emph{credal set} and is denoted by $\cset(\low F, \up F)$. $p$-boxes can be related to other models of imprecise probability, such as possibility and belief functions \cite{montes2017extreme, montes2018extreme,  troffaes2013connection}.  Their applications include engineering \cite{faes2021engineering}, modelling of risk and uncertainty \cite{rohmer2021targeted}, reliability \cite{schobi2017structural}. Recently, bivariate and multivariate generalizations have been proposed and related to imprecise copulas \cite{dolvzan2022some, liu2021copula, montes2015, omladivc2020constructing, omladivc2020final, omladivc2020full,   pelessoni2016}.
	
	In the present paper we analyze the set $\cset(\low F, \up F)$ from the viewpoint of convex analysis. We restrict to $p$-boxes on finite domains $\states\subset \RR$, whose credal sets have finite number of extreme points, making them convex polyhedra. Our main focus is a complete and systematic characterization of their extreme points. The analysis is based on the aspects of normal cones that have previously proved useful in the analysis of lower probabilities with special emphasis on probability intervals and 2-monotone lower probabilities in \cite{skulj2022normal}. Since $p$-boxes too can be considered as a  special case of lower probabilities, several general results developed there will also prove useful for the present analysis.  Convex analysis via normal cones has been utilized in some other areas of imprecise probabilities, such as in error estimation for lower previsions \cite{vskulj2019errors} and computational methods for imprecise stochastic processes \cite{skulj2020computing, skulj20uqop}. 
	
	Identification of extreme points is important, because every extreme point is a solution of the optimization problem of finding the minimal or maximal expectation of a real-valued function $h\colon \states\to \RR$ with respect to the credal set. The idea of characterization of extreme points in $p$-boxes is not entirely new. Two contributions, \citet{montes2017extreme} and \citet{montes2018extreme} address the same question by the methods of belief and possibility functions. Their particular interest is in estimating the number of extreme points and finding $p$-boxes where the maximal number is attained. A related problem was also addressed by \citet{utkin2009computing}, where the optimization problem of calculating extremal expectations with respect $p$-boxes is analyzed. Extreme points for related models of imprecise probabilities have also been analyzed in \cite{dempster2008upper} for the case of belief functions; \cite{WALLNER2007339} for general lower probabilities; \cite{miranda2003extreme} for 2-alternating capacities; and \cite{miranda2006extreme} for some other families of non-additive measures. 
	
	The approach taken in this paper is in many ways different from those proposed previously, illuminating the problem from another perspective. In the first place, we propose a systematic characterization of all possible simplicial normal cones corresponding to extreme points of $p$-boxes. These are then related to the corresponding extreme points given in terms of extremal distribution functions. Their ranges are also analyzed in terms of the ranges of their lower and upper bounds. The normal cones, and consequently the corresponding extreme points are further investigated from the adjacency point of view. In our settings, two normal cones are adjacent if they meet in common face of codimension 1, or equivalently, if they are generated by sets of generators that differ in only one element. The adjacency between normal cones is shown to be closely related to adjacency between extreme points.
	
	The paper is structured as follows. In the subsequent section the concepts of convex polyhedra and their normal cones and their applications to the models of lower probabilities. In Section~\ref{s-ncpb}, $p$-boxes are introduced and their normal cones are characterized and classified. Our main results are collected in Sections~\ref{s-ncpb} and \ref{s-eppb}. An analysis of extreme points of $p$-boxes is given in Section~\ref{s-ncpb}. In Section~\ref{s-eppb}, some new results about extreme points are proposed, describing their ranges. Next, the extreme points are related to the normal cones. Finally, the adjacency structure on the collection of normal cones is analyzed.  
	
	\section{Normal cones}\label{s-nc}
	\subsection{Normal cones and normal fans of convex polyhedra}	
	The main focus of this paper is in normal cones of convex polyhedra generated by $p$-boxes as their credal sets. In a related case of lower probabilities, a similar analysis was conducted in \cite{skulj2022normal}, where some fundamental results on convex polyhedra and their triangulations were proposed. Our analysis here will also be based on these general results, that we now list in a compact form. 
	
	A \emph{convex polyhedron} in $\allgambles$ is a bounded convex set $\csete$ with finitely many extreme points, or equivalently, an intersection of a finite number of half spaces of the form $\{ x\in \allgambles \colon x f \geqslant b_f \}$, where $f\in \allgambles$ is a vector and $b_f$ a constant. We can thus denote it as
	\begin{equation}\label{eq-convex-polyhedron}
		\csete = \{ x\in \allgambles \colon x f \geqslant b_f \text{ for all } f \in \gambleset \},
	\end{equation}
	where $\gambleset$ is a given finite collection of vectors and $\{ b_f\colon f\in \mathcal F\}$ a corresponding collection of constants. Some of the inequalities $x f\geqslant b_f$ may in fact be equalities. This case, however, can be unified with the general case by replacing an equality condition $x f = b_f$ with two inequalities, $x f \geqslant b_f$ and $x (-f) \geqslant -b_f$. This will allow us to use the simple description \eqref{eq-convex-polyhedron} throughout the text. The set $\csete'\subseteq \csete$ obtained by turning one or more inequalities $xf\ge b_f$ in \eqref{eq-convex-polyhedron} into an equality $xf = b_f$ is then a \emph{face} of $\csete$. The set of elements of $\csete$ that are not contained in any proper face is called \emph{relative interior} of $\csete$ and denoted by $\ri{\csete}$. 
	
	A convex set $C$ is called a \emph{(convex) cone} if it is closed for non-negative scalar multiplications. That is $x\in C$ implies $\alpha x\in C$ for every $\alpha\ge 0$. Often cones are generated as non-negative linear combinations of finite sets of vectors $\mathcal G$, denoted by $\cone{\mathcal G}$. Thus, 
	\begin{equation}\label{eq-cone-generated}
		\cone{\mathcal G} = \left\{ \sum_{f\in \mathcal G} \alpha_f f \colon \alpha_f \ge 0 \text{ for all } f\in \mathcal G \right\}. 
	\end{equation}
	If a cone is of the form \eqref{eq-cone-generated}, we will say that it is \emph{generated} by $\mathcal G$. 
	A cone consists of its faces and relative interior which is equal to the set of all strictly positive linear combinations of elements in $\mathcal G$:
	\begin{equation}
		\ri{\cone{\mathcal G}} = \left\{ \sum_{f\in \mathcal G} \alpha_f f \colon \alpha_f > 0 \text{ for all } f\in \mathcal G \right\}. 
	\end{equation} 
	For some point $x\in \csete$, where $\csete$ is a convex set in $\allgambles$, we define its \emph{normal cone} to be the set 
	\begin{equation}\label{eq-ncone}
		\ncone{\csete}{x} = \{ f\in \allgambles \colon x f \leqslant y f \text{ for every } y\in \csete \}.
	\end{equation}
	That is, the normal cone of $x$ is the set of all vectors $f$ for which $x = \arg\min_{y\in\csete}y f$. The following proposition will be useful in the sequel. 
	\begin{proposition}[(see \cite{gruber:07CDG}, Proposition 14.1)]
		Let $\csete$ be a convex polyhedron of the form  \eqref{eq-convex-polyhedron} and $\ncone{\csete}{x}$ its normal cone in $x$. The following propositions hold. 
		\begin{enumerate}[(i)]
			\item $\ncone{\csete}{x} = \cone{\ncone{\csete}{x}\cap \gambleset}$.
			\item If $x$ is an extreme point, then $\dim(\ncone{\csete}{x}) = \dim \allgambles$. 
		\end{enumerate}		
	\end{proposition}
	The set of all normal cones of a polyhedron forms a \emph{normal fan}, which has the property that together with cones also contains their faces and that the union of its members is the entire space $\allgambles$, which makes it a \emph{complete fan}.

	In the case where $\mathcal G$ is linearly independent set, $\cone{\mathcal G}$ is called a \emph{simplicial cone}, as the elements of $\mathcal G$ form a simplex. Simplicial cones play a central role in this paper. A general normal cone can be subdivided into simplicial cones by means of triangulations. It can be shown (see \cite{skulj2022normal}) that every normal fan can be triangulated so that it only contains simplicial cones. Such triangulation then induces a triangulation of every included normal cone into a union of simplicial cones. In \cite{skulj2022normal}, such triangulated fans are called \emph{complete normal simplicial fans}. A complete normal simplicial fan is therefore a collection of simplicial cones that together form the entire space $\allgambles$, and can be treated separately from the corresponding convex set. It is even possible that different convex sets induce the same normal fan. Moreover, the triangulations keep the property that every cone $C$ in the triangulation satisfies $C=\cone{C\cap \gambleset}$. That is, every cone is generated by gambles in $\gambleset$. 	
	The following proposition characterizes the elements of the triangulated fan. 
	\begin{proposition}\label{prop-full-triangulation}
		Let $\cone{\mathcal G}$ be an element of a complete normal simplicial fan, where $\mathcal G\subset\gambleset$. Then 
		\begin{enumerate}[(i)]
			\item $\mathcal G$ is linearly independent;
			\item for every $f\in \gambleset\backslash \mathcal G$ we have that $f\not\in \cone{\mathcal G}$;
			\item if $|\mathcal G| = \dim \allgambles$, then a convex set $\csete$ exists such that $\cone{\mathcal G}$ is its normal cone in an extreme point;
			\item a convex set $\csete$ exists such that $\cone{\mathcal G}$ is its normal cone.  
		\end{enumerate}
	\end{proposition} 
	\begin{definition}
		Let $\mathcal G\subseteq \gambleset$ be a basis of $\allgambles$, i.e. linearly independent with $|\mathcal G| = \dim \allgambles$. A cone of the form $\cone{\mathcal G}$, such that $f\not\in\cone{\mathcal G}$ for every $f\in\gambleset\backslash \mathcal G$, is called a \emph{maximal elementary simplicial cone (MESC)}. 
	\end{definition}
	In the sequel we focus on the set of possible MESCs. That is, we will analyze the properties of MESCs without reference to a specific convex polyhedra inducing them. Proposition~\ref{prop-full-triangulation} ensures that every possible MESC can be a part of triangulated normal fan.	Yet, not every collection of them can fit together. A useful tool for the analysis of possible configurations of extreme points in polyhedra is the endowment of a graph structure to the set of its extreme points. In the next section we build on this idea to generate a graph on the set of all MESCs, which can reveal some structural properties of normal fans and especially their triangulated forms.  
	
	\subsection{Adjacent cones and graph theoretical properties of maximal elementary simplicial cones}\label{ss-acgt}
	A convex polyhedron $\csete$ can be endowed a graph structure, where the extreme points are taken as vertices. An edge between two vertices is then a one dimensional face of $\csete$ with the given vertices as extreme points. The graph structure can be extended to the set of MESCs that are related to the extreme points in the sense described by Proposition~\ref{prop-full-triangulation}. 
	Here we list the most important results from \cite{skulj2022normal}. 
	\begin{definition}
		Let $\mathcal G$ and $\mathcal G'$ be two subsets of $\gambleset$, so that the corresponding cones $\cone{\mathcal G}$ and $\cone{\mathcal G'}$ are MESCs. Then the cones are said to be adjacent if they intersect in a common face of codimension 1.  
	\end{definition}
	The converse of the above corollary is not true, though. The following lemma gives the necessary and sufficient conditions for adjacency. 
	\begin{lemma}\label{lem-adjacent-cone}
		Let cones $\cone{\mathcal G}$ and $\cone{\mathcal G'}$ be MESCs such that $|\mathcal G\cap \mathcal G'| = |\mathcal G|-1 = |\mathcal G'|-1$. Let $\mathcal H$ be the hyperplane generated by $\mathcal G\cap \mathcal G'$ and $t$ its normal (non-zero) vector. Further let $f\in \mathcal G\backslash \mathcal G'$ and $f'\in \mathcal G'\backslash \mathcal G$. Then $\cone{\mathcal G}$ and $\cone{\mathcal G'}$ are adjacent if and only if $(f\cdot t)(f'\cdot t) < 0$, i.e. the scalar products with the normal vector have opposite signs.  
	\end{lemma}
	\begin{corollary}\label{cor-adjacent-extreme-points}
		Let $\csete$ be a polyhedron and $\mathcal T$ its complete normal simplicial fan. Let $C, C'\in \mathcal T$ be a pair of adjacent MESCs. Then exactly one of the following holds:
		\begin{enumerate}[(i)]
			\item An extreme point $x\in \csete$ exists such that $C, C'\subseteq \ncone{\csete}{x}$.
			\item A pair of extreme points $x, x'$ exists that lie in a common 1-dimensional face of $\mathcal T$ such that $C\subseteq \ncone{\csete}{x}$ and $C'\subseteq \ncone{\csete}{x'}$. 
		\end{enumerate} 
	\end{corollary}
	Every pair of adjacent MESCs corresponds to either a pair of adjacent extreme points or the same extreme point, which means that the two cones belong to the triangulation of its normal cone. 
	
	With $\mathbb{F}$ we now denote the set of all MESCs. This set represents all possible maximal simplicial cones of polyhedra in the form \eqref{eq-convex-polyhedron}. Let $G = (\mathbb{F}, \mathbb{E})$ be the graph with the set of vertices $\mathbb{F}$ and the set of edges $\mathbb{E} = \{ (C, C')\colon C, C'\in \mathbb{F}, C \text{ and } C' \text{ are adjacent}\}$.

	Every MESC $\cone{\mathcal G}$ defines a unique vector $x$ satisfying $xf = b_f$ for every $f\in \mathcal G$. Not every such vector however is an extreme point in $\csete$. A characterization is given by the following proposition. 
	\begin{proposition}\label{prop-cone-realization}
		A MESC $\cone{\mathcal G}$ lies within the normal cone of an extreme point $x\in \csete$ of the form \eqref{eq-convex-polyhedron} if and only if the following conditions are satisfied:
		\begin{enumerate}[(i)]
			\item $xf = b_f$ for every $f\in \mathcal G$;
			\item $xf \ge b_f$ for every $f\in \mathcal F$.
		\end{enumerate}
	\end{proposition}
	
	\subsection{Normal cones of lower probabilities}
	In \cite{skulj2022normal}, polyhedral properties of credal sets corresponding to lower probabilities and some of their special cases were analyzed through their normal cones. The polyhedral structure is obtained via \emph{credal sets} that are associated to most models of imprecise probabilities, including $p$-boxes. Credal sets are sets of probability distributions -- usually assumed to be finitely additive, or equivalently the corresponding expectation functionals, called \emph{linear previsions} $P\colon \allgambles\to \RR$. In the most general form, a credal set on a probability space $(\states, \mathcal E)$, where $\mathcal E$ is an algebra of subsets, is given by 
	\begin{equation}\label{eq-credal-set}
		\cset = \{ P \colon P(f) \ge \lpr (f) \text{ for every }f\in \gambleset\subseteq \allgambles\}.
	\end{equation}
	The above construction corresponds to a polyhedron if the set $\gambleset$ is finite. The support function $\lpr$ is often called \emph{lower prevision}. Moreover, if $\lpr(f) = \min_{P\in \cset} P(f)$ holds for all $f\in \gambleset$, the lower prevision is called \emph{coherent} and is extended by the means of \emph{natural extension} to the entire $\allgambles$ by setting:
	\begin{equation}
		\lpr(h) = \min_{P\in \cset} P(h)
	\end{equation} 
	for every $h\in \allgambles$. 
	If $\gambleset$ consists of characteristic functions $\charf{A}$ of sets in $A\in \mathcal E$, then we are talking about \emph{lower probabilities}, and a function $L$ is usually assigned, such that $L(A) = \lpr (\charf{A})$. 
	
	Let $\cset$ be credal set corresponding to a lower probability $L$ or a coherent lower prevision $\lpr$. Then for some linear prevision $P$ in (the boundary of) $\cset$, the corresponding normal cone is equivalently denoted  by either $\ncone{\cset}{P}, \ncone{L}{P}$ or $\ncone{\lpr}{P}$. In the case of lower probabilities and their special cases, normal cones are generated by collections of indicator functions $\charf{A}$ of sets in $\mathcal E$. Thus, a normal cone $\ncone{L}{P}$ is of the form  
	\begin{equation}\label{eq-nc-lp}
		\cone{\setcol{A}} := \cone{\{ \charf{A}\colon A\in \setcol{A}\backslash\{\states\}\}} + \mathrm{lin} \{ \charf{\states}\},
	\end{equation}
	where $\mathcal A = \{ A\colon P(\charf{A}) = L(A)  \}$. 
	The element $\charf{\states}$, that is the constant 1 on $\states$, can thus appear with any real valued coefficient in \eqref{eq-cone-generated}. This is because of the constraint $P(\charf{\states}) = 1$, which denotes that linear previsions in $\cset$ are normalized, and can be understood to appear in two converse inequalities $P(\charf{\states}) \ge 1$ and $P(-\charf{\states}) \ge -1$. 
	
	Coherent lower previsions are in general superadditive, i.e. $\lpr(f+g)\ge \lpr (f) + \lpr(g)$. However, as the following proposition tells, when restricted to normal cones, they become additive.
	\begin{proposition}[Normal cone additivity]\label{prop-cone-additivity}
		Take arbitrary vectors $g, h\in \ncone{\lpr}{P}$. Then $\lpr(g+h) = \lpr (g) + \lpr(h)$. 
	\end{proposition}
	Special cases of the above proposition include \emph{comonotone additivity}, a well known characterization of 2-monotone lower probabilities. 
	
	Finally, we characterize MESCs corresponding to extreme points of lower probabilities. 
	\begin{proposition}\label{prop-lprob-triangulation}
		Let $\cone{\setcol{A}}$ be a MESC in a complete normal simplicial fan. Then:  
		\begin{enumerate}[(i)]
			\item The vectors $\{ \charf{A}\colon A\in \setcol{A} \}$ are linearly independent.
			\item No equation of the form $\sum_{\setcol{A}\backslash\{\states\}}\alpha_A \charf{A}  + \alpha_\states \charf{\states} = \charf{B}$, where $B\not\in \setcol{A}$, has a solution such that $\alpha_A\ge 0$ for every $A\in \setcol{A}\backslash\{\states\}$. 	
		\end{enumerate}	
	\end{proposition}

\section{Normal cones of $p$-boxes}\label{s-ncpb}
\subsection{$p$-boxes and their credal sets}\label{ss-pbtcs}
Instead of the full structure of probability spaces, we are often concerned only with the distribution functions of specific random variables. The set of relevant events where the probabilities have to be given then shrinks considerably. In the precise case, a single distribution function $F$ describes the distribution of a random variable $X$, which gives the probabilities of the events of the form $\{ X\le x\}$. Thus $F(x) = P(X\le x)$. 

In the imprecise case, the probabilities of the above form are replaced by the corresponding lower (and upper) probabilities, resulting in sets of distribution functions called $p$-boxes \cite{ferson2003, troffaes2011probability}. A \emph{$p$-box} is a pair $(\low F, \up F)$ of distribution functions with $\low F\leqslant \up F$, where $\low F(x) = \lpr(X\leqslant x)$ and $\up F(x) = \upr(X\leqslant x)$. To every $p$-box we associate the credal set of all  distribution functions with the values between the bounds:
\begin{equation*}
	\cset(\low F, \up F) = \{ F \colon F \text{ is a distribution function}, \low F \leqslant F \leqslant \up F \}.
\end{equation*}
Clearly, $\cset(\low F, \up F)$ is a convex set of distribution functions. Conversely, since supremum and infimum of any set of distribution functions are themselves distribution functions, every set of distribution functions generates a $p$-box containing the original set.

In this paper we restrict to $p$-boxes on finite domains. Thus let $\states = \{ x_1, \ldots, x_n\}\subset \RR$, where the ordering $x_1<x_2<\cdots <x_n$ is assumed. The lower and upper bounds are then two increasing functions $\low F, \up F \colon \states \to [0, 1]$, such that $\low F \le \up F$ and $\low F(x_n)=\up F(x_n) = 1$. A description of $\cset(\low F, \up F)$ as polyhedral credal set follows. First note that every distribution function $F$ on $\states$ corresponds to a probability mass function $p_F$, such that $p_F(x_i) = F(x_i)-F(x_{i-1})$, for every $i=1, \ldots, n$, where we set $x_0 = -\infty$ and $F(x_0) = 0$. Consequently, we set $P_F$ to be the corresponding linear prevision such that $P_F(f) = \sum_{i=1}^n f(x_i)p_F(x_i)$, for every real valued function $f$ on $\states$. Denote also $A_i = \{ x_1, \ldots, x_i\}$ for every $i=1, \ldots, n$. 

The $p$-box $(\low F, \up F)$ is a special case of lower probability. The maps $\low F$ and $\up F$ correspond to the lower and the upper distribution functions: $\low F(x_i) = L(A_i)$ and $\up F(x_i) = U(A_i)$. Thus, a $p$-box can be interpreted as pair or lower and upper probabilities with the domain $\mathcal E = \{ A_i\colon i=1, \ldots, n\}$. It is sometimes beneficial to have a representation in terms of lower probabilities alone, which can be deduced by assigning $L(A^c) = 1-U(A)$. In addition to the above constraints, the non-negativity constraints $P(\charf{\{ x_i \}}) \ge 0$ need to be added, as they are not implied by $L$ defined above.

To every $p$-box $(\low F, \up F)$ a lower expectation functional $\lpr$ can be assigned through the natural extension:
\begin{equation}
	\lpr(h) = \min_{P\in \cset(\low F, \up F)}Ph.
\end{equation} 

\subsection{General structure of normal cones of $p$-boxes}\label{ss-gs}
By the characterization of $p$-boxes in terms of lower probabilities, the corresponding normal cones will be generated by the indicator functions $\charf{A_i}, \charf{A_i^c}$ and $\charf{\{ x_i \}}$ for different values of indices $i$. A distinction from general lower probability models is that the constraints on singletons $\{ x_i\}$ are only in the form $p_F(x_i)\ge 0$, i.e. the lower bounds are always zero. This is one of the properties that will have important impact on the structure of the cones. Moreover, we will restrict our analysis to maximal elementary simplicial cones -- MESCs. 

Every MESC corresponding to a $p$-box is of the form $\cone{\mathcal B}$, where $\mathcal B$ consists of sets of the form described above. We will now give  conditions for $\mathcal B$ to form a maximal elementary simplicial cone. 
\begin{corollary}[of Proposition~\ref{prop-lprob-triangulation}]\label{cor-basic-conditions}
	Let $\mathcal B$ be a collection of sets such that $\cone{\mathcal B}$ is a MESC. Then: 
	\begin{enumerate}[(i)]
		\item $\states = A_n \in \mathcal B$;
		\item $A_i\in \mathcal B\implies A_i^c\not\in \mathcal B$; 
		\item $A_i\in \mathcal B\implies \{x_{i+1}\}\not\in \mathcal B$;
		\item $A_i^c\in \mathcal B\implies \{x_{i}\}\not\in \mathcal B$.
	\end{enumerate}	
\end{corollary}
\begin{proof}
	(i) follows from the definition \eqref{eq-nc-lp} and the reasoning given there. Further, 
	(ii) follows from the fact that if $A_i\in \mathcal B$ and $A_i^c\in \mathcal B$, then $A_i\cup A_i^c = \states$, making the set linearly dependent. Conditions (iii) and (iv) follow from the fact that $A_i \cup \{ x_{i+1}\} = A_{i+1}$ or $A_{i}^c\cup \{ x_i\} = A_{i-1}^c$ respectively, which is both forbidden by Proposition~\ref{prop-lprob-triangulation}~(ii).
\end{proof}
Let $\mathcal A$ be a collection of those sets $A_i$ such that either $A_i\in \mathcal B$ or $A_i^c\in \mathcal B$. By Corollary~\ref{cor-basic-conditions}~(ii), at most one of the inclusions is possible. We denote $\mathcal A = \{ A_{i_1}, \ldots, A_{i_k}\}$, where $\{ i_j \}$ is an increasing sequence in $j$ and
\begin{equation}
	B_j = \begin{cases}
		A_{i_j}, & \text{if } A_{i_j}\in \mathcal B  \\
		A_{i_j}^c, & \text{if } A_{i_j}^c\in \mathcal B
	\end{cases}
\end{equation} 
for every $j=1, \ldots, k$. 
In addition to sets $B_j$, singleton sets $\{ x_i \}$ need to complete $\mathcal A$ to a basis. Before proceeding to details on the possible configurations, we describe the cones, assuming $\mathcal B$ given. That is, we have a collection of sets $B_{1}, \ldots , B_{k}$, and a collection of singletons $\{ x_{j_1} \}, \ldots, \{ x_{j_r} \}$. Let us also define a sign function for the sets $B_j$:
\begin{equation}
	s(B_j) = \begin{cases}
		1, & \text{if } B_j = A_{i_j} \\
		-1, & \text{if } B_j = A_{i_j}^c.
	\end{cases}
\end{equation} 
The cone can then  be decomposed as a Minkowski sum 
\begin{equation}\label{eq-minkowsky-sum}
	\cone{\mathcal B} = \cone{\{ B_{1}, \ldots, B_{k}\}} + \cone{\{ x_{j_1} \}, \ldots, \{ x_{j_r} \}}.
\end{equation}
 While the second part is simply the set of all non-negative functions with support $\{ x_{j_1}, \ldots, x_{j_r} \}$, a more detailed description of the first part follows. 

Recall that the cone is insensitive to adding a multiple of a constant, which is $\charf{\states}$, and therefore the functions $\charf{A^c}$ can be replaced by $-\charf{A}$. Any element $h\in \cone{\{ B_{1}, \ldots, B_{k}\}}$ is then of the form
\begin{equation}\label{eq-p-box-cone-element}
	h = \sum_{j=1}^k s(B_{j})\alpha_j \charf{B_{j}},
\end{equation} 
where $\alpha_j$ are non-negative constants. Moreover, $h\in \ri{\cone{\{ B_{1}, \ldots, B_{k}\}}}$ if all $\alpha_j>0$. 
\begin{lemma}
	Let a function $h\colon \states \to \RR$ be given. The following propositions are equivalent:
	\begin{enumerate}[(i)]
		\item $h\in \ri{\cone{\{ B_{1}, \ldots, B_{k}\}}}$;
		\item $h$ is $\mathcal A$-measurable and $\mathrm{sign}(h(x')-h(x)) = s(B_{j})$ whenever $x'\in B_{{j+1}}$ and $x\in B_{j}$. 
	\end{enumerate}
	(Function $\mathrm{sign}$ assigns value $1$ to positive, $-1$ to negative numbers and 0 to 0.)
\end{lemma}
\begin{proof}
	Let us first suppose that $h\in \ri{\cone{\{ B_{1}, \ldots, B_{k}\}}}$. It is then of the form \eqref{eq-p-box-cone-element}, and therefore clearly $\mathcal A$-measurable. Now take some $x'\in B_{{j+1}}$ and $x\in B_{j}$. Then we have that $h(x')-h(x)=s(B_{j})\alpha_j$, and since $\alpha_j > 0$ the sign is equal to $s(B_{j})$. 
	
	To see the converse implication, note that $\mathcal A$-measurability of $h$ implies its form \eqref{eq-p-box-cone-element}, and by the above $\alpha_j = \dfrac{h(x')-h(x)}{s(B_{j})}$, which implies non-negativity for every $j$ by the assumptions. 
\end{proof}
It remains to give detailed relation between the two cone components. 

We proceed as follows. With respect to the ordering for the sets in $\mathcal A$ according to set inclusion, we take two adjacent sets, say $A\subset A'$. They correspond to some $B_j$ and $B_{j+1}$ in one of the four possible ways: $\{ B_j, B_{j+1} \} = \{ A, A' \}, \{ B_j, B_{j+1} \} = \{ A^c, A' \},$ $\{ B_j, B_{j+1} \} = \{ A, A'^c \}$ or $\{ B_j, B_{j+1} \} = \{ A^c, A'^c \}$. Additionally, we may add $A_0 = \emptyset$ to $\mathcal A$. 
\begin{lemma}\label{lem-n-1-singletons}
	Let $A\subset A'$ be a pair of sets as described above. Then $\mathcal B$ must contain exactly $|A'\backslash A|-1$ singletons of the form $\{x\}$ for $x\in A'\backslash A$. 
\end{lemma}
\begin{proof}
In order to form a maximal simplicial cone, $\mathcal B$ must contain exactly $n$ linearly independent vectors. Consider the matrix $M$ whose rows correspond to elements of $\states$ and columns to elements of $\mathcal B$. Enumerate elements of $\mathcal B$ and let $C_j\in \mathcal B$. Then let $m_{ij}=1$ if $C_i$ is a singleton or a set from $\mathcal A$ and $x_i\in C_j$; and $-1$ if $C_i$ is a complement of an element in $\mathcal A$ and $x_i\not\in C_j$. Thus column corresponding to the complements correspond to $\charf{C_j}-\charf{\states}$. Clearly $\mathcal B$ is linearly independent if $M$ has full rank. Now consider an element $x\in A'\backslash A$. The rows corresponding to $x$ are of the following forms:
\begin{enumerate}[(i)]
	\item 1 or $-1$ at the places corresponding to sets in $\mathcal A$ containing $A'$;
	\item 0 at the places corresponding to sets in $\mathcal A$ that are contained in $A$; 
	\item 0 at the places of singletons other than $\{ x \}$ and 1 at the place of the singleton $\{ x\}$. 
\end{enumerate}
As the parts of the vectors corresponding to (i) and (ii) are identical for all $x\in A'\backslash A$, there are exactly as many different (and therefore also linearly independent) vectors corresponding to these rows, as there are singletons $\{ x \}$ that are contained in $\mathcal B$, increased by 1. The one vector added is the one with zeros at every place corresponding to (iii). So these rows are linearly independent if and only if the number of different vectors is equal to $| A'\backslash A |$. This is equivalent to stating that exactly $|A'\backslash A|-1$ singletons must be added to $\mathcal B$ to keep full rank of $M$. 
\end{proof}

Let us denote the elements of $A'\backslash A$ with $x_1\le x_2\le \ldots\le x_m$. We consider each of the four cases separately:
\begin{description}
	\item[Case 1] $\{ B_j, B_{j+1} \} = \{ A, A' \}$. By Corollary~\ref{cor-basic-conditions}~(iii), $\{x_1\}$ cannot be in $\mathcal B$. Thus exactly $m-1$ singletons $\{x_2 \}, \ldots \{ x_m \}$ must belong to $\mathcal B$. 
	\item[Case 2] $\{ B_j, B_{j+1} \} = \{ A, A'^c \}$. By Corollary~\ref{cor-basic-conditions}~(iii) and (iv), neither $\{x_1\}$ nor $\{x_m\}$ can appear in $\mathcal B$. The remaining singletons cannot ensure full rank. The conclusion is that such a pair can only appear in the case where $|A'\backslash A| = 1$. 
	\item[Case 3] $\{ B_j, B_{j+1} \} = \{ A^c, A' \}$. In this case, all singletons $\{x_1\}, \ldots, \{x_m\}$ are allowed. Yet only $m-1$ are needed, which means that this case induces $m$ different possible  configurations. 
	\item[Case 4] $\{ B_j, B_{j+1} \} = \{ A^c, A'^c \}$. A similar reasoning as above shows that exactly the singletons $\{x_1 \}, \ldots, \{ x_{m-1} \}$ must belong to $\mathcal B$ in this case. 
\end{description}
In the following sections we characterize extreme points corresponding to the above cases. Before that, let us give an example.
\begin{example}
	Let $\states = \{ 1, \ldots, 5\}$ and denote $A_i = \{ 1, \ldots, i\}$. Consider a situation where $\mathcal B$ contains sets $A_1, A_3^c, A_4, A_5=\states$. It corresponds to Case~2, where $B_j = A_1$ and $B_{j+1} = A_3^c$. Because of $|A_3\backslash A_1| > 1$, according to our analysis, this collection cannot be completed to a MESC by adding only a singleton. 
	
	Let us demonstrate that indeed no singleton can be added to $\mathcal B$ to complete it to a MESC. As $\{1\} = A_1$ and $A_1 \cup \{2\} = A_2$ and $A_4 \cup \{5\} = \states$, the only candidates to be considered are $\{3\}$ and $\{ 4\}$. Yet, $A_3^c\cup \{3\} = A_2^c$, which is not allowed by Corollary~\ref{cor-basic-conditions}(iii). It remains to analyze adding $\{ 4\}$. In this case, however, $\charf{A_4} + \charf{A_3^c} - \charf{\{ 4\}} = \charf{A_5}$, whence the set is linearly dependent. 
\end{example}

\section{Extreme points of $p$-boxes}\label{s-eppb}
Before relating normal cones to extreme points of $p$-boxes, we provide a characterization of the extreme points that to best of my knowledge is a new one. Let a $p$-box $(\low F, \up F)$ be given, and $\mathcal M(\low F, \up F)$ be the corresponding credal set containing all compatible distribution functions. 
The correspondence between the distribution functions $F$ and linear previsions is straightforward, by setting $P_F(\{x_i\}) = F(x_i)-F(x_{i-1})$, where the additional assignment $A_0 = \emptyset$ will be handy. 
\begin{lemma}\label{lem-extreme-pbox-1}
	Let $h\in \allgambles$ be given. Then some $F\in \mathcal M(\low F, \up F)$ exists such that 
	\begin{equation}\label{eq-p-box-extreme-1}
		F(x_i) \in \{ \low F(x_i), \up F(x_i), F(x_{i-1}), F(x_{i+1}) \} =: \mathcal F_i,
	\end{equation}
	and $P_F(h)=\lpr(h)$, where $\lpr$ denotes the natural extension of the $p$-box.
	
	Moreover, every extremal distribution of $\mathcal M(\low F, \up F)$ satisfies \eqref{eq-p-box-extreme-1}. 
\end{lemma}
 
\begin{proof}
	Every $h\in\allgambles$ can be (uniquely) represented in the form $h = \sum_{i=1}^n \alpha_i \charf{A_i}$. It follows that $P_F(h) =  \sum_{i=1}^n \alpha_i F(x_i)$. Take any $x_i$ and suppose that $F(x_i)\not\in \mathcal F_i$. Now consider two cases
	\begin{enumerate}[C{a}se 1)]
		\item $\alpha_i \ge 0$. Define 
		\begin{equation}
			\tilde F(x_j) = \begin{cases}
							F(x_j), & j \neq i \\
							\max\{ F(x_{i-1}), \low F(x_i) \}, & i = j.
						\end{cases}
		\end{equation}
		Clearly, $\tilde F$ preserves monotonicity and remains in $\mathcal M(\low F, \up F)$. Moreover, $\tilde F(x_i) \le F(x_i)$, and therefore $P_{\tilde F}(h) \le P_F(h)$. 
		\item $\alpha_i \le 0$. In this case we proceed in a symmetric manner, by setting
		\begin{equation}
			\tilde F(x_j) = \begin{cases}
				F(x_j), & j \neq i \\
				\min\{ F(x_{i+1}), \up F(x_i) \}, & i = j.
			\end{cases}
		\end{equation}
		Again, $\tilde F$ is monotone and belongs to $\mathcal M(\low F, \up F)$. We also have that $P_{\tilde F}(h) \le P_F(h)$.
	\end{enumerate}
	Obviously, we can use any of the cases if $\alpha_i = 0$. 
	
	Now, if $F$ is an extreme point, then it is the unique distribution function that minimizes $P_F(h)$ for every $h$ in the relative interior of its normal cone. Suppose, $F(x)\not\in \mathcal F_i$ for some $i$. Then, by the above, some $\tilde F$ would exist such that $P_{\tilde F}(h) \le P_F(h)$, which would contradict the uniqueness of the extreme point. 
\end{proof}
\begin{theorem}
	Every extreme point in $\mathcal M(\low F, \up F)$ satisfies 
	\begin{equation}\label{eq-p-box-extreme-2}
		F(x_i) \in \{ \low F(x_j) \colon j\ge i \} \cup \{ \up F(x_k) \colon k \le i \}. 
	\end{equation}
\end{theorem}
\begin{proof}
	Let $F$ be an extreme point in $\mathcal M(\low F, \up F)$ and $h\in \allgambles$ such that $F$ is the unique distribution function satisfying $P_F(h) = \lpr (h)$. Now consider some $x_i\in \states$. Then, by Lemma~\ref{lem-extreme-pbox-1}, either $F(x_i)\in \{ \low F(x_i), \up F(x_i) \}$ or $F(x_i) \in \{ F(x_{i-1}), F(x_{i+1}\}$. We only need to consider the latter case. Let $x_k, x_{k+1}, \ldots, x_{k+1}$ be a maximal subset of consequent elements of $\states$ containing $x_i$ and satisfying $F(x_{k+j}) = F(x_k)$ for every $j=0, \ldots, m$. Maximality means that $F(x_{k-1})<F(x_k)<F(x_{k+m+1})$. 
	
	Let $h=\sum_{i=1}^n \alpha_i \charf{A_i}$, which implies that 
	\begin{equation}
		P_F(h) = \sum_{i=1}^n \alpha_i F(x_i) = \sum_{i=1}^{k-1} \alpha_i F(x_i) + \sum_{i=k}^{k+m} \alpha_i F(x_i) + \sum_{i=k+m+1}^{n} \alpha_i F(x_i).
	\end{equation}
	 Using the fact that $F(x_{k+j})$ are equal for $j=0, \ldots, m$, the middle term of the latter expression can be rewritten as $F(x_k)\sum_{i=k}^{k+m} \alpha_i$. Now consider two cases:
	\begin{enumerate}[C{a}se 1)]
		\item $\sum_{i=k}^{k+m} \alpha_i \ge 0$. Then if $F(x_{k+m}) = \low F(x_{k+m})$, then $F$ satisfies \eqref{eq-p-box-extreme-2}. Otherwise, that is, if $F(x_{k+m}) > \low F(x_{k+m})$, we define
		\begin{equation}
			\tilde F(x_i) = \begin{cases}
				F(x_i), & i<k \text{ or } i > k+m; \\
				\max\{ F(x_{k-1}), \low F(x_{k+m}) \}, & k\le i \le k+m.
			\end{cases}
		\end{equation}
		Recall that $F(x_{k-1}) < F(x_k)$. 
		Clearly then, $\tilde F(x_i) < F(x_i)$ for $k\le i \le k+m$ and therefore $P_{\tilde F} (h) \le P_F(h)$, which contradicts the uniqueness of $F$ as an extreme point. This contradiction now ensures that $F(x_i) = \low F(x_{k+m})$.
		\item $\sum_{i=k}^{k+m} \alpha_i \le 0$. This case is symmetric to the first one. Thus, we either have that $F(x_k) = \up F(x_k)$, whence $F$ satisfies \eqref{eq-p-box-extreme-2}, or contrary, $F(x_k)<\up F(x_k)$. To see that the latter case again leads to contradiction, recall that $F(x_{k+j})< F(x_{k+m+1})$ and $F(x_{k+j})< \up F(x_{k})$ for every $j=0, \ldots, m$. We then define
		\begin{equation}
			\tilde F(x_i) = \begin{cases}
				F(x_i), & i<k \text{ or } i > k+m; \\
				\min\{ F(x_{k+m+1}), \up F(x_{k}) \}, & k\le i \le k+m.
			\end{cases}
		\end{equation}
		Again we have that $\tilde F(x_i) > F(x_i)$ for $k\le i \le k+m$ and therefore $P_{\tilde F} (h) \le P_F(h)$.
	\end{enumerate}
\end{proof}

\subsection{Relating extreme points to their normal cones}
Classification of cases related to normal cones in Section~\ref{ss-gs} allows us to characterize all possible extreme points of $p$-boxes according to their global and local behaviour. Recall that every normal cone is characterized by a chain $\mathcal A$ and a sign function $s$ that specifies whether a set $A\in \mathcal A$ or its complement $A^c$ belongs to the set of positive generators of the cone. The missing generators are singletons, whose possible configurations are analyzed at the end of that section. The collection of all generator sets is denoted by $\mathcal{B}$. 

Let $F$ be an extreme point whose corresponding normal cone is generated by a collection $\mathcal B$, and $P_F$ be the corresponding linear prevision. The natural extension $p$-box is denoted by $\lpr$. Take any pair $(B_j, B_{j+1})$ as described in Section~\ref{ss-gs}, and denote the elements of $A'\backslash A$ with $x_1\le x_2\le \ldots\le x_m$ and set $x_0 = \max A$. 
\begin{description}
	\item[Case 1] $\{ B_j, B_{j+1} \} = \{ A, A' \}$. Given that $A$ and $A'$ are among the generators of the cone, $P_F(A) = \lpr(A) = \low F(x_0)$ and $P_F(A') = \lpr(A') = \low F(x_m)$ must hold. In addition, singletons $\{ x_2 \}, \ldots, \{ x_m \}$ belong to $\mathcal B$. The constraints on singletons are all of the form $P(x_i) \ge 0$, hence, in the extreme point we have $0 = P_F(x_i) = F(x_i)-F(x_{i-1})$ for every $i=2, \ldots, m$. Thus 	
	\begin{equation}\label{eq-case1}
		F(x_i) = 
		\begin{cases}
			\low F(x_0),  & i = 0; \\
			 \low F(x_m), & 1\le i \le m. 
		\end{cases}
	\end{equation}
	This case is possible exactly if $\up F(x_1) \ge \low F(x_m)$, which is needed to ensure $F(x_1) \le \up F(x_1)$.  
	\item[Case 2] $\{ B_j, B_{j+1} \} = \{ A, A'^c \}$. This case is only possible if $|A'\backslash A|=1$, and in that case no other restrictions are needed. Hence, 
	\begin{equation}\label{eq-case2-1}
		F(x_0) = P_F(A) = \lpr(A) = \low F(x_0)
	\end{equation}
	and $P_F(A'^c)= \lpr(A'^c) = 1-\upr (A')$, whence
	\begin{equation}\label{eq-case2-2}
		F(x_1) = P_F(A') = \upr (A') = \up F(x_1).
	\end{equation}
	\item[Case 3] $\{ B_j, B_{j+1} \} = \{ A^c, A' \}$. In this case, $m$ different extreme points are possible, as any collection of $m-1$ singletons out of $m$ can be used as generators. Let $\{ x_k\}$ be the missing one of the singletons. First we have that $P_F(A^c) = \lpr(A^c) = 1-\upr(A) = 1-\up F(x_0) = 1-P_F(A)$, and therefore $F(x_0) = \up F(x_0)$. On the other hand we have that $F(x_m) = \low F(x_m)$. The constraints on the singletons imply that
	\begin{equation}\label{eq-case3}
		F(x_i) = 
		\begin{cases}
			 \up F(x_0), & 0 \le i \le k-1; \\
			\low F(x_m), & k \le i\le m.
		\end{cases}
	\end{equation}
	This case is possible whenever $\up F(x_0) \ge \low F(x_{k-1})$ and $\low F(x_m) \le \up F(x_k)$. 
	\item[Case 4] $\{ B_j, B_{j+1} \} = \{ A^c, A'^c \}$. This remaining case is symmetric to the first one. Reasoning similar as above gives that $F(x_0) = \up F(x_0)$ and $F(x_m) = \up F(x_m)$. The fact that $P_F(x_i) = 0$ for $i=1, \ldots m-1$ implies that $F(x_i) = F(x_{i-1})$ for those indices. Thus,we have
	\begin{equation}\label{eq-case4}
		F(x_i) = 
		\begin{cases}
			\up F(x_0), & 0 \le i \le m-1; \\
			\up F(x_m), & i = m.
		\end{cases}
	\end{equation}
	This case is possible whenever $\up F(x_0) \ge \low F(x_{m-1})$. 
\end{description}
\begin{example}\label{ex-cones-cases}
	We illustrate the above cases with the following examples. Let $\states = \{ 1, \ldots, 5\}$ and denote $A_i = \{ 1, \ldots i\}$ for $i\in \states$. 
	\begin{enumerate}
		\item  Let $p$-box $(\low F, \up F)$ be given with $\low F(i) = \frac{i}{5}$ and $\up F(i) = 1$ for every $i\in \states$. Consider the cone generated by $\mathcal B = \{ A_1, A_4, A_5, \{ 3\}, \{ 4 \} \}$. This example corresponds to Case 1 with $B_j = A_1$ and $B_{j+1} = A_4$. The corresponding extreme distribution is depicted with blue line in Fig.~\ref{fig-case1}(left). Its values are $F(1) = \low F(1), F(2) = F(3) = F(4) = \low F(4)$ and $F(5) = \low F(5)$. 
		\item Consider again the above example where $\low F$ is modified by setting $\low F(3)=\frac 45$ (see Fig.~\ref{fig-case1} (right)). Now the same extreme $F$ corresponds to two MESCs. The first one, as above, and the second one is the one generated by $\mathcal B' = \{ A_1, A_3, A_4, A_5, \{ 3\} \}$. The normal cone corresponding to the extreme point $F$ thus consists of the two adjacent cones that triangulate it. (See also Case 1A.)
		\item Let this time $\low F = (0, 0, 0.2, 0.2, 0.5, 1)$, given as a vector of values $(\low F(i))_{i=1, \ldots, 5}$, and similarly, $\up F = (0.2, 0.4, 0.6, 0.8, 1)$. Let the cone generated by $\mathcal B = \{ A_1, A_2^c, A_4^c, A_5, \{ 3\}\}$ be given. We thus have a combination of Cases 2 and 4. The corresponding extreme distribution is $F = (0, 0, 0.4, 0.4, 0.8, 1)$. The depiction of this case is given in Fig.~\ref{fig-case3}. 
		\item Finally, we give an example illustrating Case 3. Let $\mathcal B$ contain $A_1^c, A_4$ and $A_5$. According to Case 3, it must additionally contain exactly two of the three singletons $\{ 2\}, \{ 3\}$ and $\{ 4\}$, yielding three different cones, whose corresponding extreme points are depicted in Fig.~\ref{fig-case4}.
		
		The probability mass functions corresponding to the three cases are respectively $p_1 = (0.4, 0.2, 0, 0, 0.4), p_2 = (0.4, 0, 0.2, 0, 0.4)$ and $p_3 = (0.4, 0, 0, 0.2, 0.4)$. Representatives of each cone are for instance the sums of the characteristic functions of the generator sets (except $A_5 = \states$), giving $h_1 = (1, 2, 3, 3, 1), h_2 = (1, 3, 2, 3, 1)$ and $h_3 = (1, 3, 3, 2, 1)$. Let $P_i$ be the linear previsions corresponding to the above respective probability mass functions $p_i$. Then we have that $P_i(h_i) = 1.2$ and $P_j(h_i) = 1.4$ for all $i\neq j$, which clearly confirms the $p_i$ are extreme distributions corresponding to the respective normal cones. 
	\end{enumerate}
\end{example}
	
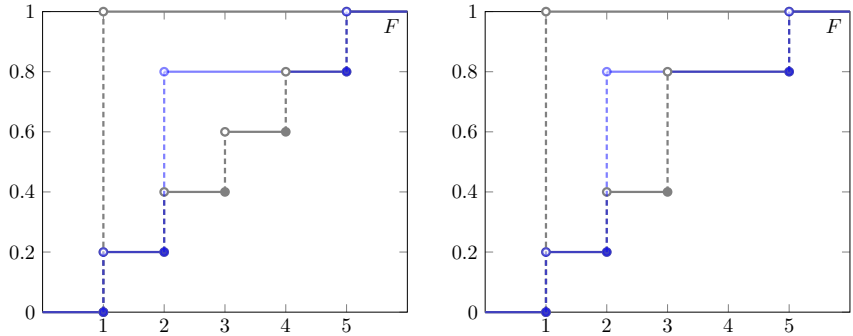
\begin{figure} 
	\begin{center}
		\begin{minipage}{.45\textwidth}
			\begin{tikzpicture}[scale=0.7]
	\begin{axis}[
		clip=false,
		jump mark left,
		ymin=0,ymax=1,
		xmin=0, xmax=6,
		xtick = { 1, ..., 5},
		ytick = {0, 1/5, 2/5, 3/5, 4/5, 1},
		every axis plot/.style={very thick},
		discontinuous,
		table/create on use/cumulative distribution/.style={
			create col/expr={\pgfmathaccuma + \thisrow{f(x)}}   
		}
		]
		\addplot [gray] table [y=cumulative distribution]{
			x f(x)
			0 0
			1 1/5
			2 1/5
			3 1/5
			4 1/5
			5 1/5
			6 0
		};
		
		\addplot [gray] table [y=cumulative distribution]{
			x f(x)
			0 0
			1 1
			2 0
			3 0
			4 0
			5 0
			6 0
		};
		
		\addplot [blue, opacity = 0.5] table [y=cumulative distribution]{
			x f(x)
			0 0
			1 1/5
			2 3/5
			3 0/5
			4 0/5
			5 1/5
			6 0
		} node[ below left, black, opacity = 1] {$F$};
		
	\end{axis}
\end{tikzpicture}
		\end{minipage}
		\begin{minipage}{.45\textwidth}
			\begin{tikzpicture}[scale=0.7]
	\begin{axis}[
		clip=false,
		jump mark left,
		ymin=0,ymax=1,
		xmin=0, xmax=6,
		xtick = { 1, ..., 5},
		ytick = {0, 1/5, 2/5, 3/5, 4/5, 1},
		every axis plot/.style={very thick},
		discontinuous,
		table/create on use/cumulative distribution/.style={
			create col/expr={\pgfmathaccuma + \thisrow{f(x)}}   
		}
		]
		\addplot [gray] table [y=cumulative distribution]{
			x f(x)
			0 0
			1 1/5
			2 1/5
			3 2/5
			4 0/5
			5 1/5
			6 0
		};
		
		\addplot [gray] table [y=cumulative distribution]{
			x f(x)
			0 0
			1 1
			2 0
			3 0
			4 0
			5 0
			6 0
		};
		
		\addplot [blue, opacity = 0.5] table [y=cumulative distribution]{
			x f(x)
			0 0
			1 1/5
			2 3/5
			3 0/5
			4 0/5
			5 1/5
			6 0
		} node[ below left, black, opacity = 1] {$F$};
		
	\end{axis}
\end{tikzpicture}		
\end{minipage}
\end{center}
	\caption{Extreme points corresponding to Example~\ref{ex-cones-cases}--1 (left) and --2 (right).}\label{fig-case1}
\end{figure}	
	
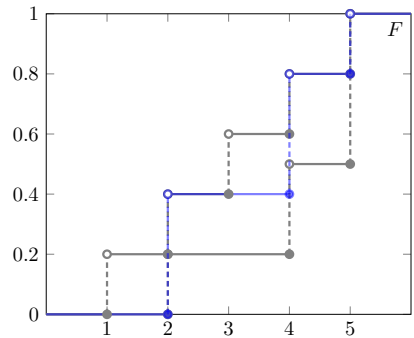
\begin{figure} 
	\begin{center}
		\begin{tikzpicture}[scale=0.7]
			\begin{axis}[
				clip=false,
				jump mark left,
				ymin=0,ymax=1,
				xmin=0, xmax=6,
				xtick = { 1, ..., 5},
				ytick = {0, 1/5, 2/5, 3/5, 4/5, 1},
				every axis plot/.style={very thick},
				discontinuous,
				table/create on use/cumulative distribution/.style={
					create col/expr={\pgfmathaccuma + \thisrow{f(x)}}   
				}
				]
				\addplot [gray] table [y=cumulative distribution]{
					x f(x)
					0 0
					1 0
					2 1/5
					3 0/5
					4 0.3
					5 0.5
					6 0
				};
				
				\addplot [gray] table [y=cumulative distribution]{
					x f(x)
					0 0
					1 0.2
					2 0.2
					3 0.2
					4 0.2
					5 0.2
					6 0
				};
				
				\addplot [blue, opacity = 0.5] table [y=cumulative distribution]{
					x f(x)
					0 0
					1 0
					2 0.4
					3 0
					4 0.4
					5 0.2
					6 0
				} node[ below left, black, opacity = 1] {$F$};
				
			\end{axis}
		\end{tikzpicture}
	\end{center}
	\caption{Extreme distribution corresponding to Example~\ref{ex-cones-cases}--3.}\label{fig-case3}
\end{figure}

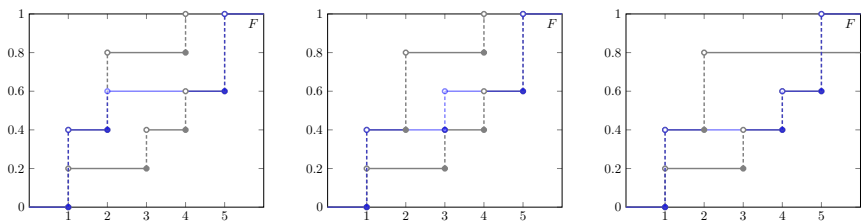
\begin{figure} 
	\begin{center}
		\begin{minipage}{.3\textwidth}
			\begin{tikzpicture}[scale=0.45]
				\begin{axis}[
					clip=false,
					jump mark left,
					ymin=0,ymax=1,
					xmin=0, xmax=6,
					xtick = { 1, ..., 5},
					ytick = {0, 1/5, 2/5, 3/5, 4/5, 1},
					every axis plot/.style={very thick},
					discontinuous,
					table/create on use/cumulative distribution/.style={
						create col/expr={\pgfmathaccuma + \thisrow{f(x)}}   
					}
					]
					\addplot [gray] table [y=cumulative distribution]{
						x f(x)
						0 0
						1 .2
						2 0
						3 .2
						4 0.2
						5 0.4
						6 0
					};
					
					\addplot [gray] table [y=cumulative distribution]{
						x f(x)
						0 0
						1 0.4
						2 0.4
						3 0
						4 0.2
						5 0
						6 0
					};
					
					\addplot [blue, opacity = 0.5] table [y=cumulative distribution]{
						x f(x)
						0 0
						1 0.4
						2 0.2
						3 0
						4 0
						5 0.4
						6 0
					} node[ below left, black, opacity = 1] {$F$};
					
				\end{axis}
			\end{tikzpicture}
		\end{minipage}
		\begin{minipage}{.3\textwidth}
			\begin{tikzpicture}[scale=0.45]
				\begin{axis}[
					clip=false,
					jump mark left,
					ymin=0,ymax=1,
					xmin=0, xmax=6,
					xtick = { 1, ..., 5},
					ytick = {0, 1/5, 2/5, 3/5, 4/5, 1},
					every axis plot/.style={very thick},
					discontinuous,
					table/create on use/cumulative distribution/.style={
						create col/expr={\pgfmathaccuma + \thisrow{f(x)}}   
					}
					]
					\addplot [gray] table [y=cumulative distribution]{
						x f(x)
						0 0
						1 .2
						2 0
						3 .2
						4 0.2
						5 0.4
						6 0
					};
					
					\addplot [gray] table [y=cumulative distribution]{
						x f(x)
						0 0
						1 0.4
						2 0.4
						3 0
						4 0.2
						5 0
						6 0
					};
					
					\addplot [blue, opacity = 0.5] table [y=cumulative distribution]{
						x f(x)
						0 0
						1 0.4
						2 0
						3 0.2
						4 0
						5 0.4
						6 0
					} node[ below left, black, opacity = 1] {$F$};
					
				\end{axis}
			\end{tikzpicture}		
		\end{minipage}
	\begin{minipage}{.3\textwidth}
		\begin{tikzpicture}[scale=0.45]
			\begin{axis}[
				clip=false,
				jump mark left,
				ymin=0,ymax=1,
				xmin=0, xmax=6,
				xtick = { 1, ..., 5},
				ytick = {0, 1/5, 2/5, 3/5, 4/5, 1},
				every axis plot/.style={very thick},
				discontinuous,
				table/create on use/cumulative distribution/.style={
					create col/expr={\pgfmathaccuma + \thisrow{f(x)}}   
				}
				]
				\addplot [gray] table [y=cumulative distribution]{
					x f(x)
					0 0
					1 .2
					2 0
					3 .2
					4 0.2
					5 0.4
					6 0
				};
				
				\addplot [gray] table [y=cumulative distribution]{
					x f(x)
					0 0
					1 0.4
					2 0.4
					3 0
					4 0
					5 0
					6 0
				};
				
				\addplot [blue, opacity = 0.5] table [y=cumulative distribution]{
					x f(x)
					0 0
					1 0.4
					2 0
					3 0
					4 0.2
					5 0.4
					6 0
				} node[ below left, black, opacity = 1] {$F$};
				
			\end{axis}
		\end{tikzpicture}
	\end{minipage}
	\end{center}
	\caption{Extreme points corresponding to Example~\ref{ex-cones-cases}--4, with the excluded singletons $\{2\}$ (left), $\{3\}$ (middle) and $\{4\}$ (right).}\label{fig-case4}
\end{figure}	

\subsection{Adjacency structure of the normal cones}
In this section we identify possible adjacency relations between the normal cones of the form $\cone{\mathcal B}$ whose general structure is described in Section~\ref{ss-gs}. Recall that $\mathcal B$ is formed by a chain of sets $\mathcal A$, whose members or their complements are contained in $\mathcal B$, and a set of singletons. The possible relations between them are described in Section~\ref{ss-gs}. As follows from Lemma~\ref{lem-adjacent-cone}, the cones adjacent to $\cone{\mathcal B}$ are obtained by replacing each member of $\mathcal B$ by another element. 

Let us first consider the singletons in $\mathcal B$. To every singleton $\{ x\}$ a pair of adjacent sets $A\subset A'$ exists so that $x\in A'\backslash A$. As before, we will enumerate the elements of $A'\backslash A$ by $x_1, \ldots, x_m$, consistently with the order in $\states$. Again we consider the cases as in Section~\ref{ss-gs}. For each case, a singleton $\{ x_i\}$ is omitted from $\mathcal B$ and replaced by either another singleton in $A'\backslash A$ or a set $B$ so that $A\subset B \subset A'$. We cannot replace it by other elements without violating the requirement by Lemma~\ref{lem-n-1-singletons}. First, the candidates are identified, and then Lemma~\ref{lem-adjacent-cone} is used to check whether the candidate forms an adjacent cone. In case where more than one candidate is possible, the actual conical structure is deduced from the values of $\low F$ and $\up F$. 
\begin{description}
	\item[Case 1A] $\{ B_j, B_{j+1} \} = \{ A, A' \}$. Let $\{ x_i\}$ be the excluded singleton. Note that the singletons for this case are exactly $\{ x_2\}, \ldots, \{ x_m\}$. The only candidate to replace $\{ x_i\}$ among singletons is thus $\{ x_1\}$. Yet, it would result in a cone structure not consistent with any case of the classification from Section~\ref{ss-gs}. 
	
	Another possibility is either a set $B$ or its complement. Taking $B=A\cup\{x_1, \ldots, x_{i-1}\}$ would result in the case $(A, B, A')$ where both pairs $(A, B)$ and $(B, A')$ together with the remaining singletons forms a structure consistent with the classification. 
	
	The only set $C^c$ such that $(A, C^c, A')$ remains consistent with the classification is the case where $C=A\cup \{ x_1\}$. The induced pairs $(A, C^c)$ and $(C^c, A')$ clearly comply with the classifications. 
	
	Now we have two candidates that must be additionally verified whether they indeed form adjacent cones. Let us calculate the normal vector $t$ to the hyperplane generated by the remaining sets. In fact we are only interested in its values on $A'\backslash A$. For every remaining singleton $\{x_j\}$ we have that $t\charf{\{ x_j \}} = 0$, which implies $t(x_j) = 0$. Moreover, we have that $t\charf{A} = t\charf{A'} = 0$, whence $t\charf{A'\backslash A} = t\charf{A'} - t\charf{A} = 0$. Now only $t(x_1)$ and $t(x_i)$ are non-zero, and by the above, we have that $t(x_1) + t(x_i) = 0$. Thus, we can take $t(x_1) = -1$ and $t(x_i) = 1$. For the excluded vector we then have that $t\charf{\{ x_i \}} = t(x_i) = 1$. For the first adjacency candidate we obtain $t\charf{B} = t(x_1) = -1$, which has the opposite sign from $t\charf{\{ x_i \}}$, and consequently $B$ is an adequate candidate. However, in the case of $C^c$ we have $t\charf{C^c} = t(\charf{\states}-\charf{C}) = -t(x_1) = 1$, concluding that $C^c$ does not satisfy the requirements. Thus, we have shown that we have a unique adjacent cone for this case. The unique adjacent cone is then of the form Case 1B (see below).
	
	If now $F$ is a distribution function corresponding to the initial case. The distribution function $F'$ corresponding to the adjacent cone equals 
	\begin{equation}\label{eq-case-1a}
		F'(x_k) = 
		\begin{cases}
			 \low F(x_{i-1}), & 1\le k\le i-1; \\
			 F(x_k), & \text{otherwise}.
		\end{cases}
	\end{equation}
	Indeed, $\up F(x_1) \ge \low F(x_m)$, which follows from the condition for Case 1, clearly implies $\up F(x_1) \ge \low F(x_m) \ge \low F(x_{i-1})$, which allows the above assignment, and therefore confirms that this is the only adjacent extreme point for this case. Notice also that $F' \le F$. 
	
	\item[Case 2A] $\{ B_j, B_{j+1} \} = \{ A, A'^c \}$. In this case there are no singletons from $A'\backslash A$ in $\mathcal B$ to replace.

	\item[Case 3A] $\{ B_j, B_{j+1} \} = \{ A^c, A' \}$. The set $\mathcal B$ corresponding to this case contains any $m-1$ singletons out of $\{ x_1\}, \ldots, \{ x_m\}$. Let $\{ x_j\}$ be the one not present, and $\{ x_i\}$ the one we exclude. Similar reasoning as above gives $t(x_i) = 1$ and $t(x_j) = -1$ as the only non-zero components for the normal vector, where again, $t\charf{\{ x_i \}} = 1$. 
	
	Now, the first candidate to replace $\{ x_i\}$ is the singleton $\{x_j\}$, and since $t\charf{\{ x_j \}} = -1$, it corresponds to an adjacent cone. 
	
	The remaining candidates are sets $B = A\cup\{ x_1, \ldots, x_{i\vee j-1}\}$ and $C^c$, where $C = A\cup\{ x_1, \ldots, x_{i\wedge j}\}$. We have that $t\charf{B} = t(x_{i\wedge j})$ and $t\charf{C^c} = t(x_{i\vee j})$. So in case where $i>j$, $B$ fits requirements and if $i<j$, then $C^c$ does. So, in this case we have two possible adjacent cones. The one actually corresponding to a particular case is identified as follows. 
	
	Let $F$ correspond to the extreme point corresponding to our initial case and $F'$ to the extreme point corresponding to the adjacent cone. 
	By the analysis of Case 3, we have that 
	\begin{equation}
		F(x_k) = 
		\begin{cases}
			\up F(x_0), & 1 \le k \le j-1, \\
				\low F(x_m), & j \le k\le m.
		\end{cases}
	\end{equation}	
	Let us first consider the case $i<j$. The replacement of $\{ x_i\}$ with $\{ x_j\}$ would result in $F'(x_k) = \low F(x_m)$ for every $k\ge i$. This is only possible if $\low F(x_m)\le \up F(x_i)$, whereas in the opposite case where  $\low F(x_m)\ge \up F(x_i)$, replacing $\{x_i\}$ with $C^c$ results in $F'(x_k) = \up F(x_i)$ for $i\le k \le j-1$, and $F'(x_k) = F(x_k)$ elsewhere. Clearly, only one of the possibilities for $F'$ lies withing the bounds. The adjacent cones correspond to the Cases 3A and 3B (second variant) respectively.   
	
	In the case where $i>j$ replacing $\{x_i\}$ with $\{x_j\}$ is possible in the case where $\low F(x_i) \le \up F(x_0)$. The resulting distribution function $F'$ then satisfies $F'(x_k) = \up F(x_0)$ for $j\le k \le i-1$ and $F'(x_k) = F(x_k)$ elsewhere. In the case where $\low F(x_i) \ge \up F(x_0)$, $\{ x_i\}$ must be replaced by $B$, which results in $F(x_k) = \low F(x_{i-1})$ for $j\le k \le i-1$, and $F'(x_k) = F(x_k)$ elsewhere. The adjacent cones correspond to the Case 3A and 3B (first variant) respectively. 
	
	Let us summarize. The adjacent extreme point depends on the following conditions: 
	\begin{description}
		\item[$i < j.$] Two adjacent cones are possible in this case, depending on:  
		\begin{description}
			\item[$\low F(x_m)\le \up F(x_i).$] In this case, $\{ x_i\}$ is replaced by $\{ x_j\}$, and the corresponding extreme distribution is
			\begin{equation}\label{eq-case-3a-1}
				F'(x_k) = 
				\begin{cases}
					\low F(x_m), & i\le k \le j-1; \\
					F(x_k), & \text{otherwise}.
				\end{cases}
			\end{equation}
			Hence, $F'\ge F$ holds. 
			\item[$\low F(x_m)\ge \up F(x_i).$] In this case, $\{ x_i\}$ is replaced by $C^c$ where $C = A\cup\{ x_1, \ldots, x_{i}\}$, and the corresponding extreme distribution is
			\begin{equation}\label{eq-case-3a-2}
				F'(x_k) = 
				\begin{cases}
					\up F(x_i), & i \le k \le j-1; \\
					F(x_k), & \text{otherwise}. 
				\end{cases}
			\end{equation}
			Again we have $F'\ge F$. 
		\end{description}
		\item[$i > j.$] We have again two adjacent cones possible, depending on:
		 \begin{description}
		 	\item[$\low F(x_i)\le \up F(x_0).$] In this case, $\{ x_i\}$ is replaced by $\{ x_j\}$, and the corresponding extreme distribution is
		 	\begin{equation}\label{eq-case-3a-3}
		 		F'(x_k) = 
		 		\begin{cases}
		 			\up F(x_0), & j\le k \le i-1; \\
		 			F(x_k), & \text{otherwise}.
		 		\end{cases}
		 	\end{equation}
		 	Hence, $F'\le F$ holds. 
		 	\item[$\low F(x_i)\ge \up F(x_0).$] In this case, $\{ x_i\}$ is replaced by $B = A\cup\{ x_1, \ldots, x_{i-1}\}$, and the corresponding extreme distribution is
		 	\begin{equation}\label{eq-case-3a-4}
		 		F'(x_k) = 
		 		\begin{cases}
		 			\low F(x_{i-1}), & j \le k \le i-1; \\
		 			F(x_k), & \text{otherwise}. 
		 		\end{cases}
		 	\end{equation}
		 	Again we have $F'\le F$. 
		 \end{description}
	\end{description}
	
	\item[Case 4A] $\{ B_j, B_{j+1} \} = \{ A^c, A'^c \}$. In this case the excluded $\{ x_i\}$ cannot be replaced by a singleton. The only possible candidates are the sets $B^c$, where $B = A\cup\{x_1, \ldots, x_{i}\}$ and $C = A'\backslash \{ x_m\}$. Scalar products with the normal vectors then reveal that $B^c$ fits, while $C$ does not. Here again we have a unique adjacent normal cone corresponding to the Case 4B (second variant). 
	
	Let $F$ and $F'$ correspond to the initial and the adjacent cone respectively. Then we have that 
	\begin{equation}\label{eq-case-4a}
		F'(x_k) = 
		\begin{cases}
			\up F(x_{i-1}), & i-1\le k\le m-1; \\
			F(x_k), & \text{otherwise}.
		\end{cases}
	\end{equation}
	The condition allowing the above assignment is $\up F(x_{i-1}) \ge \low F(x_{m-1})$, which is implied by the conditions $F(x_0) \ge \low F(x_{m-1})$ and $\up F(x_0)\le \up F(x_{i-1})$ that follow from the requirements for Case 4 and monotonicity of $\up F$. This additionally confirms that $F'$ is the unique extreme distribution adjacent to $F$. We can also notice that $F'\ge F$. 
\end{description}
It remains to analyze adjacent cones for the cases where sets from $\mathcal A$ are excluded from $\mathcal B$. The cases where sets from $\mathcal A$ are replaced by singletons are already covered by the Cases 1A--4A, analyzed above.  Thus we will focus to cases where those sets are replaced by sets of the same type, that is either $B$ or $C^c$, where $B$ or $C$ together with $\mathcal A$ forms a chain. Our analysis will this time consider triples $(B_j, B_{j+1}, B_{j+2})$ whose components are obtained by an adjacent chain $A\subset A' \subset A''$ as its elements or their complements. We will thus consider 4 cases with two variants each. Denote the elements of $A'\backslash A = \{ x_1, \ldots, x_{m'}\}$ and $A''\backslash A' = \{ x_{m'+1}, \ldots, x_{m}\}$. 
\begin{description}
	\item[Case 1B] $\{ B_j, B_{j+1}, B_{j+2} \} = \{ A, A', A'' \}$. Clearly, there is no other $B$ that could replace $A'$ and still satisfy the requirements described by general Cases 1--4. The only possible non-singleton set is a complement $C^c$ of some $A\subset C \subset A''$. The general situation described by Case 2 reveals that the only candidate is $C=A\cup \{ x_1\}$, that does turn out to generate an adjacent cone. Reversing the analysis also covers the subvariant $\{ B_j, B_{j+1}, B_{j+2} \} = \{ A, A'^c, A'' \}$. 
	
	The other possibility comes from Case 1A, where $A'$ is replaced by $\{ x_{m'+1}\}$. The actual adjacent cone and the corresponding extreme distribution is analyzed as follows. 	
	Let $F$ be the distribution function corresponding to the initial cone \eqref{eq-case1} and $F'$ the one corresponding to the adjacent cone. We thus have: 
	\begin{equation}
		F(x_k) = 
		\begin{cases}
			\low F(x_0), & k = 0; \\
			\low F(x_{m'}), & 1\le k \le m'; \\
			\low F(x_{m}), & m'+1\le k \le m. 
		\end{cases}
	\end{equation}	
	
	Which one of the possible candidates is an adjacent cone is analyzed as follows.

	The case where $A'$ is replaced by $\{ x_{m'+1}\}$ leads to the Case 1A, where 
	\begin{equation}\label{eq-case-1b-1}
		F'(x_k) = \low F(x_m) \text{ for }1\le k\le m.
	\end{equation}
	 Such $F'$ only lies within the bounds if $\low F(x_m)\le \up F(x_1)$. In this case $F'\ge F$. 
	 
	 If the relation $\low F(x_m)\ge \up F(x_1)$ holds, then $C^c$ can only replace $A'$ where $C=A\cup\{x_1\}$. This case corresponds to our second variant of Case 1B. The corresponding $F''$ then satisfies 	 
	 \begin{equation}\label{eq-case-1b-2}
	 		F''(x_k) = 
			\begin{cases}
 				\up F(x_1), & 1\le k \le m'; \\
				F(x_k), & \text{otherwise}.
			\end{cases}
	\end{equation}
	Because $F$ lies within the bounds, we can deduce that $\up F(x_1)\ge \low F(x_{m'})$, whence $F''\ge F$ follows again. 
	
	The subvariant $\{ B_j, B_{j+1}, B_{j+2} \} = \{ A, A'^c, A'' \}$ is only adjacent to the primary variant, even though $A'^c$ could be replaced by $\{ x_{m'+1}\}$ that would fall under the Case A1. However, the analysis there, showed that the corresponding cone is not adjacent. Another way to observe this, is to notice that condition $\up F(x_1)\ge \low F(x_{m'})$ required by the subvariant implies that required for the primary variant, which is therefore always viable unique neighbour. 
	
	Notice that in this and other cases where the adjacent cone is not unique, we have a borderline situation where both cases are possible. The two cases then correspond to two elements of a triangulation of the normal cone corresponding to a single extreme point. 
	
	\item[Case 2B] $\{ B_j, B_{j+1}, B_{j+2} \} = \{ A, A', A''^c \}$. This case is only possible if $A''=A'\cup\{ x_m\}$. Again, no other $B$ fits, by the analysis of the general Case 2, so we consider the complements $C^c$. The only possibility is again $C = A\cup\{ x_1\}$, that indeed generates an adjacent cone. The subvariant $\{ B_j, B_{j+1}, B_{j+2} \} = \{ A, A'^c, A''^c \}$ is again covered by reversing the above analysis. 
	Let $F$ and $F'$ be the extremal distributions corresponding to the original cone and its neighbour respectively. By the above, we have that $m=m'+1$ and it follows from the general structure that
	\begin{equation}
		F(x_k) = 
		\begin{cases}
			\low F(x_0), & k = 0; \\
			\low F(x_{m'}), & 1\le k \le m'; \\
			\up F(x_{m}), & k = m. 
		\end{cases}
	\end{equation}		
 	Further we have that 
 	\begin{equation}\label{eq-case-2b}
 		F'(x_k) = 
 		\begin{cases}
 			\low F(x_0), & k = 0; \\
 			\up F(x_1), & 1\le k \le m'; \\
 			\up F(x_{m}) = F(x_k), & k = m. 
 		\end{cases}
 	\end{equation}
 	Since $\low F(x_{m'}) = \low F(x_{m-1}) \le \up F(x_1)$, this adjacent extreme distribution $F'$ always lies within the bounds if $F$ does and vice versa. Also, $F'\ge F$ holds.  
	
	\item[Case 3B] $\{ B_j, B_{j+1}, B_{j+2} \} = \{ A^c, A', A'' \}$. Removing $A'$ leaves the singletons $\{ x_1\}, \ldots, \{ x_m\}$ with exception of $\{x_{m'+1}\}$ and $\{ x_j\}$ for some $1\le j \le m'$. Again there are no possibilities for $B$ replacing $A'$. The candidate of the form $C^c$ is $C=A\cup\{ x_1, \ldots, x_j\}$, that turns out adequate. The reverse analysis again covers the subvariant  $\{ B_j, B_{j+1}, B_{j+2} \} = \{ A^c, A'^c, A'' \}$.
	
	Another possible adjacent cone is the one described in Case 3A. Let us now analyze the actual cones based on the values of $\low F$ and $\up F$. First we have that 
	\begin{equation}
		F(x_k) = 
		\begin{cases}
			\up F(x_0), & 0\le k \le j-1; \\
			\low F(x_{m'}), & j\le k \le m'; \\
			\low F(x_{m}), & m'+1\le k \le m. 
		\end{cases}
	\end{equation}
	First of the two candidates to replace $A'$ is $C^c$. In this case we obtain the following extreme distribution:
	\begin{equation}\label{eq-case-3b-1}
		F'(x_k) = 
		\begin{cases}
			\up F(x_0), & 0\le k \le j-1; \\
			\up F(x_{j}), & j\le k \le m'; \\
			\low F(x_{m}), & m'+1\le k \le m. 
		\end{cases}
	\end{equation}
	Such $F$ is only feasible if $\up F(x_j) \ge \low F(x_{m'})$, which is implied by $F$ lying between the bounds; and $\up F(x_j) \le \low F(x_m)$ to ensure monotonicity of $F'$. Relation $F'\ge F$ is then also implied. 
	
	In the remaining case, $\{ x_{m'+1}\}$ replaces $A'$ and the corresponding extreme distribution is 
	\begin{equation}\label{eq-case-3b-2}
		F(x_k) = 
		\begin{cases}
			\up F(x_0), & 0\le k \le j-1; \\
			\low F(x_{m}), & j\le k \le m. 
		\end{cases}
	\end{equation}
	This case is of course only feasible if $\low F(x_m) \le \up F(x_j)$. Again we have that $F'\ge F$. These cases again cover all possible values of $\low F$ and $\up F$. 
	
	The second subvariant $\{ B_j, B_{j+1}, B_{j+2} \} = \{ A^c, A'^c, A'' \}$ again allows two possible adjacent cones obtained by replacing $A'^c$. The first one is the first subvariant of this Case 3B and the second one leads to Case 3A. Let us start with the starting extreme distribution. Let $\{x_j\}$ be the missing singleton according to Case 3. We have:
	\begin{equation}
		F(x_k) = 
		\begin{cases}
			\up F(x_0), & 0\le k \le m'-1; \\
			\up F(x_{m'}), & m'\le k \le j-1; \\
			\low F(x_{m}), & j\le k \le m. 
		\end{cases}
	\end{equation}
	The first adjacent cone is obtained by replacing $A'^c$ with $B=A'\cup \{ x_{m+}, \ldots, x_{j-1}\}$ resulting in the extreme distribution
	\begin{equation}\label{eq-case-3b-3}
		F'(x_k) = 
		\begin{cases}
			\up F(x_0), & 0\le k \le m'-1; \\
			\low F(x_{j-1}), & m'\le k \le j-1; \\
			\low F(x_{m}), & j\le k \le m. 
		\end{cases}
	\end{equation}
	The starting distribution lies within the bounds, which implies that $\low F(x_{j-1}) \le \up F(x_{m'})$, which is required for the above distribution to also lie within the bounds. Moreover, to ensure its monotonicity, we must have that $\low F(x_{j-1}) \ge \up F(x_{0})$. We have that $F'\le F$. 
	
	The remaining case is obtained by replacing $A'^c$ with $\{ x_{m'}\}$, resulting in the distribution 
	\begin{equation}\label{eq-case-3b-4}
		F''(x_k) = 
		\begin{cases}
			\up F(x_0), & 0\le k \le m'-1; \\
			\up F(x_{0}), & m'\le k \le j-1; \\
			\low F(x_{m}), & j\le k \le m. 
		\end{cases}
	\end{equation}
	This case is clearly feasible exactly if $\low F(x_{j-1}) \le \up F(x_{0})$. Again we have that $F''\le F$.

	\item[Case 4B] $\{ B_j, B_{j+1}, B_{j+2} \} = \{ A^c, A', A''^c \}$. This case is only possible if $A''=A'\cup \{ x_m\}$. Removing $A'$ leaves $\{ x_m\}$ and $\{ x_j\}$ omitted for some $1\le j \le m'$. Again, no $B$ can replace $A'$, and taking $C=A\cup\{ x_1, \ldots, x_j\}$, gives $C^c$ as a candidate that again proves adequate. The subvariant $\{ B_j, B_{j+1}, B_{j+2} \} = \{ A^c, A'^c, A''^c \}$ is again covered by reversing the analysis. 
	
	Let $F$ be the extremal distribution corresponding to the first subvariant. We have:
	\begin{equation}
		F(x_k) = 
		\begin{cases}
			\up F(x_0), & 0\le k \le j-1; \\
			\low F(x_{m'}), & j\le k \le m'; \\
			\up F(x_{m}), & k=m. 
		\end{cases}
	\end{equation}
	The only adjacent cone is the one belonging to the second subvariant, whose extremal distribution is 
	\begin{equation}\label{eq-case-4b-1}
		F'(x_k) = 
		\begin{cases}
			\up F(x_0), & 0\le k \le j-1; \\
			\up F(x_{j}), & j\le k \le m'; \\
			\low F(x_{m}), & j\le k \le m. 
		\end{cases}
	\end{equation}
	Clearly, this distribution lies within the bounds whenever $\up F(x_j) \ge \low F(x_{m'})$ which is exactly in the case where $F$ is feasible. We have that $F'\ge F$. 

	Yet, the reverse analysis shows that in addition to the first subvariant, the second one can have another adjacent cone, the one corresponding to Case 4A. Let this time $F$ correspond to the extremal distribution corresponding to the second subvariant. Thus, 
	\begin{equation}
		F(x_k) = 
		\begin{cases}
			\up F(x_0), & 0\le k \le m'-1; \\
			\up F(x_{m'}), & m'\le k \le m-1; \\
			\up F(x_{m}), & k=m. 
		\end{cases}
	\end{equation}
	The first adjacent cone is obtained by replacing $A'^c$ with $B = A\cup \{ x_1, \ldots, x_{m-1}\}$ inducing the extreme distribution
	\begin{equation}\label{eq-case-4b-2}
		F'(x_k) = 
		\begin{cases}
			\up F(x_0), & 0\le k \le m'-1; \\
			\low F(x_{m-1}), & m'\le k \le m-1; \\
			\up F(x_{m}), & k=m. 
		\end{cases}
	\end{equation}
	It is only feasible in the case where $\low F(x_{m-1}) \ge \up F(x_0)$. We have that $F' \le F$. 
	
	The remaining case is obtained by replacing $A'^c$ with $\{ x_{m'}\}$. It results in 
	\begin{equation}\label{eq-case-4b-2}
		F''(x_k) = 
		\begin{cases}
			\up F(x_0), & 0\le k \le m-1; \\
			\up F(x_{m}), & k=m. 
		\end{cases}
	\end{equation}
	It is feasible exactly if $\low F(x_{m-1}) \le \up F(x_0)$. Again, $F''\le F$. 
\end{description}
\begin{example}\label{ex-adjacent1}
	Let $\states = \{ 1, \ldots, 5\}$ and denote $A_i = \{ 1, \ldots, i\}$ for $i\in \states$, and $p$-box $(\low F, \up F)$ be given with a vector of its values on $\states$. Consider the following examples of pairs of adjacent cones and corresponding extreme points. 
	\begin{enumerate}
		\item Let $\low F(i) = \frac{i}{5}$ and $\up F(i) = 1$ for every $i\in \states$. 
		Consider the cone generated by $\mathcal B = \{ A_1, A_3, A_4, A_5, \{ x_3\}\}$ (see Case 1). The corresponding extreme point is the distribution function $F = (0.2, 0.6, 0.6, 0.8, 1)$ (see Fig. \ref{fig-ex1}, left). Let us now generate an adjacent cone by replacing $A_3$ with another set. This corresponds to Case 1B, with $(A_1, A_3, A_4)$ and $m' = 3, m = 4$. Inequality $\low F(x_4) < \up F(x_1)$ implies that replacing $A_3$ with $\{ x_4\}$ results in an adjacent cone, which is thus generated by  $\mathcal B' = \{ A_1, A_4, A_5, \{ x_3\}, \{ x_4\}\}$ and corresponds to the extreme point $F'=(0.6, 0.6, 0.6, 0.8, 1)$ (see Fig. \ref{fig-ex1}, right). 
		\item Let $p$-box $(\low F, \up F)$ be given by $\low F = (0.2, 0.2, 0.6, 0.8, 1)$ and $\up F = (0.4, 0.4, 1, 1, 1)$. 
		Consider the cone generated by $\mathcal B = \{ A_1, A_2^c, A_3, A_4, A_5\}$, which corresponds to the extreme distribution function $F = (0.2, 0.4, 0.6, 0.8, 1)$ (see Fig. \ref{fig-ex2}, left). 
				
		Let us now generate an adjacent cone by replacing $A_3$ with another set. This corresponds to Case 3B, with $(A_2^c, A_3, A_4)$. Thus we have $x_0 = 2, x_m' = 3, x_m = 4$ and $x_j = 3$. Further, we have $\low F(x_m) = \low F(4) = 0.6 < \up F(x_j) = \up F(3) = 1$. This implies that $\{ 4\}$ is the right replacement. The adjacent cone is thus generated by  $\mathcal B' = \{ A_1, A_2^c, A_4, A_5, \{4\}\}$ corresponding to the extreme point $F'=(0.2, 0.4, 0.8, 0.8, 1)$ (see Fig. \ref{fig-ex2}, right). 
		\item Let $p$-box $(\low F, \up F)$ be given with $\low F = (0, 0, 0.8, 0.8, 1)$ and $\up F = (0.2, 0.4, 0.8, 0.8, 1)$. 
		Consider the cone generated by $\mathcal B = \{ A_1^c, A_3, A_4, A_5, \{ x_3\}\}$, which corresponds to the distribution function $F = (0.2, 0.4, 0.8, 0.8, 1)$ (see Fig. \ref{fig-ex3}, left). 
		
		An adjacent cone obtained by replacing $\{ 3\}$ with another set corresponds to Case 3A, with $(A_1^c, A_3)$ and $\{ 2\}$ left out. We then have $x_j = 2, x_i = 3, x_m = 3$ and $x_0 = 1$. Since $i>j$ and $\up F(x_0) = \up F(1) = 0.2 \le 0.8 = \low F(3) = \low F(x_i)$,  $\{ 3\}$ is replaced by $A_2 = A_{i-1}$, resulting in the adjacent cone generated by  $\mathcal B' = \{ A_1^c, A_2, A_3, A_4, A_5\}$ and corresponding to the extreme point $F'=(0.2, 0.2, 0.8, 0.8, 1)$ (see Fig. \ref{fig-ex3}, right).
	\end{enumerate}
	
\end{example}
\begin{figure}
	\begin{center}
	\begin{minipage}{.45\textwidth}
	\begin{tikzpicture}[scale=0.7]
		\begin{axis}[
			clip=false,
			jump mark left,
			ymin=0,ymax=1,
			xmin=0, xmax=6,
			xtick = { 1, ..., 5},
			ytick = {0, 1/5, 2/5, 3/5, 4/5, 1},
			every axis plot/.style={very thick},
			discontinuous,
			table/create on use/cumulative distribution/.style={
				create col/expr={\pgfmathaccuma + \thisrow{f(x)}}   
			}
			]
			\addplot [gray] table [y=cumulative distribution]{
				x f(x)
				0 0
				1 1/5
				2 1/5
				3 1/5
				4 1/5
				5 1/5
				6 0
			};
			
			\addplot [gray] table [y=cumulative distribution]{
				x f(x)
				0 0
				1 1
				2 0
				3 0
				4 0
				5 0
				6 0
			};
			
			\addplot [blue, opacity = 0.5] table [y=cumulative distribution]{
				x f(x)
				0 0
				1 1/5
				2 2/5
				3 0/5
				4 1/5
				5 1/5
				6 0
			} node[ below left, black, opacity = 1] {$F$};
		
		\end{axis}
	\end{tikzpicture}
\end{minipage}
\begin{minipage}{.45\textwidth}
		\begin{tikzpicture}[scale=0.7]
			\begin{axis}[
		clip=false,
		jump mark left,
		ymin=0,ymax=1,
		xmin=0, xmax=6,
		xtick = { 1, ..., 5},
		ytick = {0, 1/5, 2/5, 3/5, 4/5, 1},
		every axis plot/.style={very thick},
		discontinuous,
		table/create on use/cumulative distribution/.style={
			create col/expr={\pgfmathaccuma + \thisrow{f(x)}}   
		}
		]
		\addplot [gray] table [y=cumulative distribution]{
			x f(x)
			0 0
			1 1/5
			2 1/5
			3 1/5
			4 1/5
			5 1/5
			6 0
		};
		
		\addplot [gray] table [y=cumulative distribution]{
			x f(x)
			0 0
			1 1
			2 0
			3 0
			4 0
			5 0
			6 0
		};

		\addplot [blue, opacity = 0.5] table [y=cumulative distribution]{
			x f(x)
			0 0
			1 3/5
			2 0/5
			3 0/5
			4 1/5
			5 1/5
			6 0
		} node[ below left, black, opacity = 1] {$F'$};
	\end{axis}
	\end{tikzpicture}
\end{minipage}
\end{center}
	\caption{$p$-boxes with the adjacent extreme points from Example \ref{ex-adjacent1}--1 depicted.} \label{fig-ex1}
\end{figure}

\begin{figure}
	\begin{center}
		\begin{minipage}{.45\textwidth}
			\begin{tikzpicture}[scale=0.7]
				\begin{axis}[
					clip=false,
					jump mark left,
					ymin=0,ymax=1,
					xmin=0, xmax=6,
					xtick = { 1, ..., 5},
					ytick = {0, 1/5, 2/5, 3/5, 4/5, 1},
					every axis plot/.style={very thick},
					discontinuous,
					table/create on use/cumulative distribution/.style={
						create col/expr={\pgfmathaccuma + \thisrow{f(x)}}   
					}
					]
					\addplot [gray] table [y=cumulative distribution]{
						x f(x)
						0 0
						1 1/5
						2 0/5
						3 2/5
						4 1/5
						5 1/5
						6 0
					};
					
					\addplot [gray] table [y=cumulative distribution]{
						x f(x)
						0 0
						1 2/5
						2 0/5
						3 3/5
						4 0/5
						5 0/5
						6 0
					};
					
					\addplot [blue, opacity = 0.5] table [y=cumulative distribution]{
						x f(x)
						0 0
						1 1/5
						2 1/5
						3 1/5
						4 1/5
						5 1/5
						6 0
					} node[ below left, black, opacity = 1] {$F$};
					
				\end{axis}
			\end{tikzpicture}
		\end{minipage}
		\begin{minipage}{.45\textwidth}
			\begin{tikzpicture}[scale=0.7]
				\begin{axis}[
					clip=false,
					jump mark left,
					ymin=0,ymax=1,
					xmin=0, xmax=6,
					xtick = { 1, ..., 5},
					ytick = {0, 1/5, 2/5, 3/5, 4/5, 1},
					every axis plot/.style={very thick},
					discontinuous,
					table/create on use/cumulative distribution/.style={
						create col/expr={\pgfmathaccuma + \thisrow{f(x)}}   
					}
					]
					\addplot [gray] table [y=cumulative distribution]{
						x f(x)
						0 0
						1 1/5
						2 0/5
						3 2/5
						4 1/5
						5 1/5
						6 0
					};
					
					\addplot [gray] table [y=cumulative distribution]{
						x f(x)
						0 0
						1 2/5
						2 0/5
						3 3/5
						4 0/5
						5 0/5
						6 0
					};
					
					\addplot [blue, opacity = 0.5] table [y=cumulative distribution]{
						x f(x)
						0 0
						1 1/5
						2 1/5
						3 2/5
						4 0/5
						5 1/5
						6 0
					} node[ below left, black, opacity = 1] {$F'$};
				\end{axis}
			\end{tikzpicture}
		\end{minipage}
	\end{center}
	\caption{$p$-boxes with the adjacent extreme points from Example \ref{ex-adjacent1}--2 depicted.} \label{fig-ex2}
\end{figure}

\begin{figure}
	\begin{center}
		\begin{minipage}{.45\textwidth}
			\begin{tikzpicture}[scale=0.7]
				\begin{axis}[
					clip=false,
					jump mark left,
					ymin=0,ymax=1,
					xmin=0, xmax=6,
					xtick = { 1, ..., 5},
					ytick = {0, 1/5, 2/5, 3/5, 4/5, 1},
					every axis plot/.style={very thick},
					discontinuous,
					table/create on use/cumulative distribution/.style={
						create col/expr={\pgfmathaccuma + \thisrow{f(x)}}   
					}
					]
					\addplot [gray] table [y=cumulative distribution]{
						x f(x)
						0 0
						1 0/5
						2 0/5
						3 4/5
						4 0/5
						5 1/5
						6 0
					};
					
					\addplot [gray] table [y=cumulative distribution]{
						x f(x)
						0 0
						1 1/5
						2 1/5
						3 2/5
						4 1/5
						5 0/5
						6 0
					};
					
					\addplot [blue, opacity = 0.5] table [y=cumulative distribution]{
						x f(x)
						0 0
						1 1/5
						2 1/5
						3 2/5
						4 0/5
						5 1/5
						6 0
					} node[ below left, black, opacity = 1] {$F$};
					
				\end{axis}
			\end{tikzpicture}
		\end{minipage}
		\begin{minipage}{.45\textwidth}
			\begin{tikzpicture}[scale=0.7]
				\begin{axis}[
					clip=false,
					jump mark left,
					ymin=0,ymax=1,
					xmin=0, xmax=6,
					xtick = { 1, ..., 5},
					ytick = {0, 1/5, 2/5, 3/5, 4/5, 1},
					every axis plot/.style={very thick},
					discontinuous,
					table/create on use/cumulative distribution/.style={
						create col/expr={\pgfmathaccuma + \thisrow{f(x)}}   
					}
					]
				\addplot [gray] table [y=cumulative distribution]{
					x f(x)
					0 0
					1 0/5
					2 0/5
					3 4/5
					4 0/5
					5 1/5
					6 0
				};
				
				\addplot [gray] table [y=cumulative distribution]{
					x f(x)
					0 0
					1 1/5
					2 1/5
					3 2/5
					4 1/5
					5 0/5
					6 0
				};
					
					\addplot [blue, opacity = 0.5] table [y=cumulative distribution]{
						x f(x)
						0 0
						1 1/5
						2 0/5
						3 3/5
						4 0/5
						5 1/5
						6 0
					} node[ below left, black, opacity = 1] {$F'$};
				\end{axis}
			\end{tikzpicture}
		\end{minipage}
	\end{center}
	\caption{$p$-boxes with the adjacent extreme points from Example \ref{ex-adjacent1}--3 depicted.} \label{fig-ex3}
\end{figure}
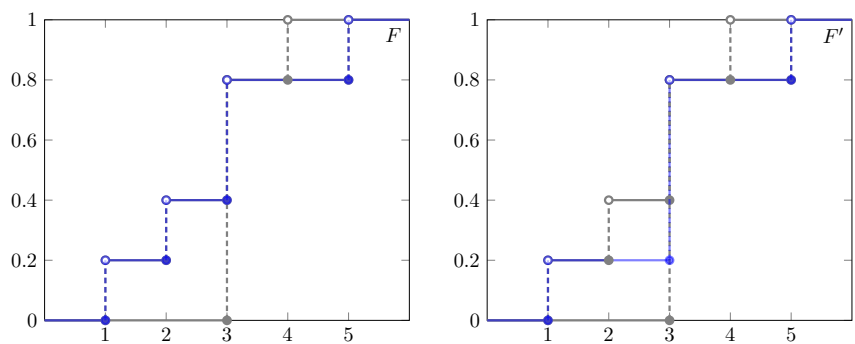

\section{Conclusions}
The main contribution of this manuscript is a complete characterization and classification of extreme points of $p$-boxes. This is achieved by utilizing general results on normal cones corresponding to polyhedral sets of probability distributions, and particularly those of lower probability models. The classification not only allows describing extreme points in relation to vectors whose expectations they minimize. In addition, the adjacency structure is also explored that in principle allows finding all extreme points by building a regular graph of normal cones connected by the analyzed adjacencies. 

The utilization of normal cones has proved useful in the analysis of the polyhedral structure and particularly of extreme points for various models. A similar complete analysis has been conducted for the case of probability intervals in \cite{skulj2022normal}. In future, similar analysis could be applied to variety of models, such as generalized $p$-boxes, $p$-boxes in bivariate or multivariate settings, and also to the study of imprecise copulas. 

\section*{Acknowledgement}
The author acknowledges the financial support from the Slovenian Research Agency (research core funding No. P5-0168).

\bibliography{references_all}
\end{document}